\documentclass[reqno,12pt]{amsart}
\usepackage{latexsym}
\usepackage{amssymb,amsthm,amsmath}
\usepackage[dvips]{graphicx}

\usepackage{graphicx}

\usepackage{amssymb,amsmath,amsthm,latexsym}

\newcommand{\R}{\mathbb{R}}

\newcommand{\eps}{\varepsilon}

\newcommand{\Sfrac}[2]{{ \textstyle \frac{#1}{#2}}}

\newcommand{\intall}{\int_{\R}}
\newcommand{\inthalf}{\int_{\R_+}}

\newtheorem{definition}{Definition}
\newtheorem{thm}{Theorem}
\newtheorem{lemma}{Lemma}
\newtheorem{prop}{Proposition}

\newtheorem{corollary}{Corollary}

\newtheorem{remark}[thm]{Remark}

\newcounter{zd}

\newcounter{zdr}[subsection]

\def\pa{\partial}
\def\ve{\varepsilon}
\def\cal{\mathcal}

\numberwithin{equation}{section}

\numberwithin{thm}{section}
\numberwithin{cor}{section}
\numberwithin{lemma}{section}
\numberwithin{prop}{section}
\numberwithin{definition}{section}

\begin{document}

$\;$

\vspace{0.3in}

\begin{center}
{\Large \bf Singular solutions of a fully nonlinear 2x2 system of conservation laws}
\end{center}

\medskip

\vspace{1in}

\centerline{ \large{Henrik Kalisch} \footnote{
            \texttt{Department of Mathematics, University of Bergen; e-mail: henrik.kalisch@math.uib.no}
            }
            \\
\large{Darko Mitrovic} \footnote{
            \texttt{Department of Mathematics, University of Montenegro; e-mail: matematika@t-com.me}
            }$^{,}$\footnote{\texttt{The corresponding author}.}
}

\vspace{0.5in}

{\bf Abstract:} Existence and admissibility of $\delta$-shock
solutions is discussed for the non-convex strictly
hyperbolic system of equations
\begin{align*}
\pa_t u + \pa_x \big( \Sfrac{u^2+v^2}{2} \big) &=0\\
\pa_t v +\pa_x(v(u-1))&=0.
\end{align*}
The system is fully nonlinear, i.e. it is nonlinear with respect to
both unknowns, and it does not admit the classical Lax-admissible
solution for certain Riemann problems. By introducing complex-valued
corrections in the framework of the weak asymptotic method, we show
that a compressive $\delta$-shock solution resolves such Riemann
problems. By letting the approximation parameter tend to zero, the
corrections become real-valued and the solutions can be seen to fit
into the framework of weak singular solutions defined in \cite{DSH}.
Indeed, in this context, we can show that every $2\times 2$ system
of conservation laws admits $\delta$-shock solutions.

 \vspace{0.3in}

\emph{Keywords:} \ Conservation laws, Riemann problem, singular
solutions, weak asymptotics, MHD.

\emph{AMS classification:} \ 35L65;  35L67; 76W05.

\newpage

\section{Introduction}
The main subject of this paper is a system of conservation laws
appearing in the study of plasmas. The system is known as the {\it
Brio} system and has the form
\begin{equation}
\label{brio1}
\begin{split}
\pa_t u + \pa_x\big(\Sfrac{u^2+v^2}{2} \big)&=0,\\
\pa_t v +\pa_x(v(u-1))&=0.
\end{split}
\end{equation} The system is strictly
hyperbolic; it is genuinely nonlinear at $\{(u,v):\;u\in \R, \;v>0
\}$ and $\{(u,v):\;u\in \R, \;v<0 \}$, but not on the whole of $\R^2$.
The system was introduced in \cite{brio} and thoroughly considered in
\cite{Lfl2}. There, it was found that for certain initial data no
solution consisting of the Lax-admissible elementary waves (shock
and rarefaction waves) exists. In \cite{Lfl2}, Riemann problems for
\eqref{brio1} were compared to Riemann problems for the system
\begin{equation}
\label{brio1'}
\begin{split}
\pa_t u + \pa_x \big( \Sfrac{u^2}{2} \big) &=0\\
\pa_t v +\pa_x(v(u-1))&=0.
\end{split}
\end{equation}
Numerical computations of appropriate viscous profiles for \eqref{brio1}
and \eqref{brio1'} demonstrated surprising similarities.
In the same paper, it was shown that certain Riemann problems
for \eqref{brio1'} admit $\delta$-shock wave solutions.
However, the same fact could not be
established for any Riemann problem corresponding to \eqref{brio1}.
In the present work, we aim to resolve the question of existence of
$\delta$-shock wave solutions of \eqref{brio1}, and the question of
physical justifiability of such solutions to the Riemann problem
associated to \eqref{brio1}. We remark that for
\eqref{brio1'}, if the $\delta$ distribution is a part if the
solution then it is adjoined to the function $v$ (with respect to
which the system is linear). However, in the case of system
\eqref{brio1}, our investigation shows that it is more natural for the
$\delta$ distribution to be a part of the function $u$.


The study of singular solutions of systems of conservation laws was
initiated by Korchinski \cite{Korchinski} and Keyfitz and Kranzer
\cite{KK1,KK2}. In the last few years, interest in the topic has
grown, and a sample of results may be found in \cite{CL, DM1,
ercole, Lfl2, huang1, huang2, joseph, veque, LX, DMMN, MNarma, SZ,
TZZ, Y}. One convenient tool for constructing singular solutions is
the method weak asymptotics. This method has been used recently to
understand the evolution of nonlinear waves in scalar conservation
laws as well as interaction and formation of $\delta$-shock waves in
the case of a triangular system of conservation laws
\cite{DM1,DM2,DSH}.
We refer the reader to \cite{Omel} and the references contained therein
for further applications of the weak asymptotics method.

In the present work, we introduce an extension of the weak
asymptotics method to the case where complex-valued corrections are
considered for the approximate solutions. Even though the imaginary
parts of the solutions so constructed vanish in an appropriate
limit, it appears that considering complex-valued weak asymptotic
solutions significantly extends the range of possible singular
solutions.

It appears that the weak asymptotic method has so far only been used to
construct singular solutions of systems for which the flux functions
were {\em linear} with respect to the unknown function which
contains $\delta$-distribution. In contrast, note that the flux
$(f(u,v),g(u,v))=\left((u^2+v^2)/2,v(u-1)\right)$ associated with
the system \eqref{brio1} is nonlinear in both $u$ and $v$, and none
of the existing methods yield singular solutions of this system.
Thus it appears that the use of complex-valued corrections is
essential in the construction of singular solutions for
\eqref{brio1}.



Let us next define what we mean by complex-valued weak asymptotic
solution, and highlight some methods to restrict the notion of
solution with the goal of obtaining uniqueness. First we define a
vanishing family of distributions.
\begin{definition}
\label{defwas}
Let $f_{\ve}(x) \in {\cal D}'(\R)$ be a family of
distributions depending on $\eps\in (0,1)$,
We say that $f_{\ve} = o_{{\cal D}'}(1)$
if for any test function $\phi(x)\in \mathcal{D}(\R)$, the estimate
$$
\langle f_{\ve},\phi \rangle=o(1),  \ \ {\rm as}\ \ \eps\to 0
$$
holds.
\end{definition}
The estimate on the right-hand side is understood in the
usual Landau sense. Thus we may say that a family of distributions
approach zero in the sense defined above if for a given test
function $\phi$, the pairing $\langle f_{\ve},\phi \rangle$
converges to zero as $\ve$ approaches zero.
For families of distributions $f_{\ve}(x,t)$, we write $f_{\ve} =
o_{{\cal D}'}(1)\subset {\cal D}'(\R)$ if the estimate above holds
uniformly in $t$. More succinctly, we require that
\begin{equation*}
\langle f_{\ve}(\cdot,t),\varphi \rangle \leq C_T g(\eps) \mbox{  for }
t \in [0,T],
\end{equation*}
where the function $g$ depends on the test function $\varphi(x,t)$
and tends to zero as $\eps\to 0$,
and where $C_T$ is a constant depending only on $T$.
We define weak asymptotic solutions to a general system
of two conservation laws
\begin{equation}
\label{gensystem}
\begin{split}
\pa_t u + \pa_x f(u,v)  =&  0,       \\
\pa_t v + \pa_x g(u,v) =& 0,
\end{split}
\end{equation}
as follows.
\begin{definition}
\label{defbg1} We say that the families of smooth complex-valued
distributions $(u_\eps)$ and $(v_\eps)$ represent a
weak asymptotic solution to \eqref{gensystem} if there exist
real-valued distributions $u,v\in C(\R_+;{\cal D}'(\R))$, such that
for every fixed $t\in \R_+$
$$
u_\eps\rightharpoonup u, \ \ v_\eps\rightharpoonup v \ \ {\rm
as} \ \ \eps\to 0,
$$ in the sense of distributions in  ${\cal D}'(\R)$, and
\begin{equation}
\label{gendef} \left.
\begin{array}{r @{\quad = \quad}l }
\pa_t u_\eps + \pa_x f(u_\eps,v_\eps)  &  {o}_{{\cal D}'}(1),       \\
\pa_tv_\eps + \pa_x g(u_\eps,v_\eps) & {o}_{{\cal D}'}(1).
\end{array}
\right\}
\end{equation}
\end{definition}
It is evident that this definition requires some additional assumptions
of the fluxes $f$ and $g$. In particular, $f$ and $g$ must have an extension
into the complex plane. One may for instance restrict to fluxes that
are real-analytic, though in principle a wider class of fluxes is possible.
The main issue in the requirement on the fluxes, and indeed with this method
of constructing solutions is the question of uniqueness.
For example, by adding a constant term of order ${\cal O}(\eps)$
to any weak asymptotic solution,
one immediately obtains two different weak asymptotic
solutions which correspond to the same solution if a more
restrictive concept is used.

One way to narrow the class of solution candidates is to require
distributional solutions to satisfy the equations in a stronger
sense than the one defined in Definition \ref{defbg1}. This approach
entails substituting them into \eqref{brio1}, and to check directly
whether the equations are satisfied. This strategy involves the
problem of multiplication of singular distributions. The problem of
taking products of singular distributions was overcome by Danilov
and Shelkovich in \cite{DSH} in a rather elegant way. In their work,
the weak asymptotic solution is constructed such that the terms that
do not have a distributional limit cancel in the limit as $\eps$
approaches zero. As a result, it is not necessary to include
singular terms in the definition of the weak solution. Thus, the
problem of multiplication of distributions is automatically
eliminated, and the class of possible solution is significantly
reduced.

There are also several other reasonable ways to multiply
Heaviside and Dirac distributions.
In \cite{col1, DMLM, huang, Vol}, a number of definitions of weak solutions
of \eqref{gensystem} are introduced.
Among the latter approaches, we emphasize the measure-type solution
concept introduced in \cite{DMLM, huang}. Moreover, the framework
from \cite{huang} yields uniqueness of solutions if an additional
condition of Oleinik-type is required, and that is probably the only
work so far which obtains a uniqueness result for arbitrary initial
data in a class of distributional solutions weak enough to allow
delta distributions. However, uniqueness has also been obtained for
special classes of initial data by LeFloch in \cite{Lfl3}, and by
Nedeljkov \cite{MNarma}.

We remark that a systematic study of multiplication of distributions
problem is investigated in the Colombeau algebra framework
\cite{col1, ober1, NPS}. In these works, problems of the type
considered here are also investigated.
Actually, Definition \ref{defbg1} can be understood as a
variant of appropriate definitions in \cite{CO, MNmmas, MNqam}.
The main difference is that in the present case, a solution is found pointwise
with respect to $t\in \R_+$,
and it is required that the distributional limit of the weak asymptotic solution
be a distribution. The latter is not necessary in the framework of the
Colombeau algebra though it may be tacitly assumed.

The plan of the present paper is as follows. We will provide a
review of the definition of weak singular solutions from \cite{DSH}
in Section 2. It turns out that a somewhat more general statement is
appropriate here. Moreover, it will be proved that any $2\times 2$
system of hyperbolic conservation admits singular solutions of this
type. In Section 3, weak asymptotic solutions of the Brio system are
found. The results of that section are very important since they
represent a justification of the concept introduced in Section 2
which will be applied in the Section 4. In that section, it is shown
that the limit of the weak asymptotic solutions satisfy the equation
in the sense of Definition \ref{def-gen}. Also, an adaptation of the
Lax admissibility concept is proposed which provide physically
sustainable solutions to corresponding Riemann problems. The final
section is the Appendix where we consider other possibilities for
existence of $\delta$-shock solutions.

\section{Generalized weak solutions}
In this section, the definition of weak singular solutions of a
$2\times 2$ system of conservation laws provided in \cite{DSH} is
reviewed.
Indeed, we shall show that any $2\times 2$ systems of the form
\begin{align*}
\pa_t u + \pa_x f(u,v)  =&  0,       \\
\pa_t v + \pa_x g(u,v) =& 0,
\end{align*}
admits $\delta$-type solution in the framework introduced in \cite{DSH}.
While the definition in \cite{DSH} is given only for
solutions singular in the second variable, while assuming that the
flux functions $f$ and $g$ are linear in the second variable, it
appears that the definition can actually be made more general.
Suppose $\Gamma =\{\gamma_i~|~i \in I \}$ is a graph in the closed
upper half plane, containing Lipschitz continuous arcs $\gamma_i$,
$i\in I$, where $I$ is a finite index set. Let $I_0$ be the subset
of $I$ containing all indices of arcs that connect to the $x$-axis,
and let $\Gamma_0 = \{x^0_k~|~k \in I_0 \}$ be the set of initial
points of the arcs $\gamma_k$ with $k \in I_0$. Define the singular
part by $\alpha(x,t) \delta(\Gamma) = \sum_{i \in I}\alpha_i(x,t)
\delta(\gamma_i)$. Let $(u,v)$ be a pair of distributions, where $v$
is represented in the form
$$
v(x,t)=V(x,t)+\alpha(x,t)\delta(\Gamma),
$$
and where $u,V\in L^\infty(\R \times \R_+)$.
Finally, the expression $\frac{\pa \varphi(x,t)}{\pa {\rm l}}$
denotes the tangential derivative of a function $\varphi$ on the graph
$\gamma_i$, and $\int_{\gamma_i}$ connotes the line integral over
the arc $\gamma_i$.
\begin{definition}
{\bf a)} 
The pair of distributions
$u$ and $v=V+\alpha(x,t)\delta(\Gamma)$
are called a generalized $\delta$-shock wave solution of system
\eqref{gensystem} with the initial data
$U_0(x)$ and $V_0(x)+ \sum_{I_0}\alpha_k(x_0^k,0)\delta\big( x-x_k^0 \big)$
if the integral identities
\begin{align}
\label{m1-g1} &\inthalf \! \! \intall \left( u\pa_t\varphi +f(u,V)
\pa_x\varphi\right)~dxdt + \intall U_0(x)\varphi(x,0)~dx = 0,\\
\label{m2-g1}
&\inthalf \! \! \intall \left(V\pa_t\varphi+ g(u,V)\pa_x\varphi\right)~dxdt\\
&\qquad + \sum\limits_{i\in I}\int_{\gamma_i}\alpha_i(x,t)
\Sfrac{\pa \varphi(x,t)}{\pa {\bf l}} \, + \int_{\R} V_0(x)
\varphi(x,0)~dx + \sum \limits_{k\in I_0} \alpha_k(x_k^0,0)
\varphi(x_k^0,0) = 0, \nonumber
\end{align}
hold for all test functions $\varphi\in {\cal D}(\R\times \R_+)$.
\end{definition}
\setcounter{definition}{0}
The next definition concerns the similar situation where the
singular solution is contained in $u$, and $v$ is a regular distribution.
Thus we assume the representation
$$
u(x,t)=U(x,t)+\alpha(x,t)\delta(\Gamma),
$$
where now
$U,v\in L^\infty(\R \times \R_+)$, and $\alpha(x,t)\delta(\Gamma)$
is defined as before.
\begin{definition}\label{def-gen}
{\bf b)}
The pair of distributions
$u=U+\alpha(x,t)\delta(\Gamma)$ and $v$
is a generalized $\delta$-shock wave solution of system \eqref{gensystem}
with the initial data
$U_0(x) + \sum_{I_0}\alpha_k(x_0^k,0) \delta\big( x-x_k^0 \big)$ and $V_0(x))$
if the integral identities
\begin{align}
\label{m1-g2} &\inthalf \! \! \intall \left( U\pa_t\varphi +f(U,v)
\pa_x\varphi\right)~dxdt\\
&\qquad\qquad\qquad+\sum\limits_{i\in
I}\int_{\gamma_i}\alpha_i(x,t)\Sfrac{\pa \varphi(x,t)}{\pa {\bf
l}}+\intall U_0(x)\varphi(x,0)dx + \sum \limits_{k\in I_0}
\alpha_k(x_k^0,0) \varphi(x_k^0,0)
=0,\nonumber\\
\label{m2-g2} &\inthalf \! \! \intall \left(v\pa_t\varphi
+ g(U,v)\pa_x\varphi\right)~dxdt
+ \intall V_0(x)\varphi(x,0)~dx = 0,
\end{align}
hold for all test functions $\varphi\in {\cal D}(\R\times \R_+)$.
\end{definition}
This definition may be interpreted as an extension of the
classical weak solution concept. Moreover, as noticed in e.g.
\cite{AD}, the definition is consistent with the
concept of measure solutions \cite{DMLM, huang}.

Definition \ref{def-gen} is quite general, allowing a combination of initial
steps and delta distributions; but its effectiveness is already
demonstrated by considering the Riemann problem with a single jump.
Indeed, for this configuration it can be shown that a $\delta$-shock
wave solution exists for any $2\times 2$ system of conservation
laws.
Consider the Riemann problem for \eqref{gensystem} with initial data
$u(x,0) = U_0(x)$ and  $v(x,0) = V_0(x)$,
where
\begin{equation}
\label{rieman-data}
U_0(x)=\begin{cases} u_1, &x<0\\
u_2, & x>0
\end{cases}, \quad
V_0(x)=\begin{cases} v_1, &x<0\\
v_2, & x>0
\end{cases}.
\end{equation}
Then, the following theorem holds.
\begin{thm}\label{thm-cnl}{\bf a)}
If $u_1\neq u_2$ then the pair of distributions
\begin{eqnarray}\label{sol-a1}
u(x,t) & = & U_0(x-ct), \\
\label{sol-a2}
v(x,t) & = & V_0(x-ct) + \alpha(t) \delta(x-ct),
\end{eqnarray}
where
$$
c=\frac{[f(u,V)]}{[u]}=\frac{f(u_2,v_2)-f(u_1,v_1)}{u_2-u_1}, \ \mbox{ and } \
\alpha(t)=(c[V]-[g(u,V)])t,
$$
represents the $\delta$-shock wave solution of \eqref{gensystem}
with initial data $U_0(x)$ and $V_0(x)$
in the sense of Definition \ref{def-gen} a).
\end{thm}
\setcounter{thm}{0}
\begin{thm}
{\bf b)} If $v_1\neq v_2$ then the pair of distributions
\begin{eqnarray}\label{sol-b1}
u(x,t) & = & U_0(x-ct) + \alpha(t) \delta(x-ct), \\
\label{sol-b2}
v(x,t) & = & V_0(x-ct),
\end{eqnarray}
where
$$
c=\frac{[g(U,v)]}{[v]}=\frac{g(u_2,v_2)-g(u_1,v_1)}{v_2-v_1}, \ \
\alpha(t)=(c[U]-[f(U,v)])t
$$
represents the $\delta$-shock solution of \eqref{gensystem} with
initial data $U_0(x)$ and $V_0(x)$ in the sense of Definition
\ref{def-gen} b).
\end{thm}
\begin{proof}
We will prove only the first part of the theorem as the
second part can be proved analogously. We immediately see that
$u$ and $v$ given by \eqref{sol-a1} and \eqref{sol-a2}
satisfy \eqref{m1-g1} since $c$ is given exactly by the
Rankine-Hugoniot condition derived from that system. By
substituting $u$ and $v$ into \eqref{m2-g1},
we get after standard transformations:
\begin{align*}
\int_{\R_+}\left(-c[V]+[g(u,V)]\right)\varphi(ct,t)~dt
-\int_{\R_+}\alpha'(t)\varphi(ct,t)~dt = 0.
\end{align*}
From here and since $\alpha(0)=0$, the conclusion follows
immediately.
\end{proof}

As the solution framework of Definition \ref{def-gen} is very weak,
one might expect non-uniqueness issues to arise.
This is indeed the case, and the proof of the following proposition
is an easy exercise.
\begin{prop}\label{non-unique}
System \eqref{gensystem} with the zero initial data:
$u|_{t=0}=v|_{t=0}=0$ admits $\delta$-shock solutions of the form:
\begin{align*}
u(x,t)=0, \ \ v(x,t)=\beta\delta(x-c_1t)-\beta\delta(x-c_2 t),
\end{align*}for arbitrary constants $\beta$, $c_1$ and $c_2$.
\end{prop}

At the moment, we do not have a general concept for resolving such
and similar non-uniqueness issues. In the case of the Brio system
which we shall consider in the sequel, we are also not able to
obtain uniqueness, but we can prove that there always exists a
physically reasonable solution to the corresponding Riemann problem.

Finally, let us remark, that it is of course possible and reasonable
to give a definition along the lines of Definition \ref{def-gen} which allows
for simultaneous concentration effects in both unknowns $u$ and $v$.
In this case, a generalized $\delta$-shock wave solution of \eqref{gensystem}
with the initial data
$U_0(x)+ \sum_{I_0}\alpha_k(x_0^k,0)\delta\big( x-x_k^0 \big)$ and $V_0(x)+ \sum_{I_0}\beta_k(x_0^k,0)\delta\big( x-x_k^0 \big)$
would have the form
$u=U+\alpha(x,t)\delta(\Gamma)$ and $v=V+\beta(x,t)\delta(\Gamma)$,
and satisfy
\begin{align}
\label{m1-g11} &\inthalf \! \! \intall \left( U\pa_t\varphi +f(U,V)
\pa_x\varphi\right)~dxdt\\ &\qquad + \sum\limits_{i\in I}\int_{\gamma_i}\alpha_i(x,t)
\Sfrac{\pa \varphi(x,t)}{\pa {\bf l}} + \intall U_0(x)\varphi(x,0)~dx+ \sum \limits_{k\in I_0} \alpha_k(x_k^0,0)
\varphi(x_k^0,0) = 0,\nonumber\\
\label{m2-g12}
&\inthalf \! \! \intall \left(V\pa_t\varphi+ g(U,V)\pa_x\varphi\right)~dxdt\\
&\qquad + \sum\limits_{i\in I}\int_{\gamma_i}\beta_i(x,t)
\Sfrac{\pa \varphi(x,t)}{\pa {\bf l}} \, + \int_{\R} V_0(x)
\varphi(x,0)~dx + \sum \limits_{k\in I_0} \beta_k(x_k^0,0)
\varphi(x_k^0,0) = 0, \nonumber
\end{align}
for all test functions $\varphi\in {\cal D}(\R\times \R_+)$. An
example of such a situation can be found in \cite{chrom}.

\section{Weak asymptotics for the Brio system}
In this section, we shall construct weak asymptotic solutions for
the Riemann problem associated to the Brio system \eqref{brio1}, and
then show that the weak asymptotic solution converges to the
generalized weak solution to the system in the sense of Definition
\ref{def-gen}. This construction is very important since the fact
that it is possible to find a sequence of smooth approximating
solutions to \eqref{brio1}, \eqref{rieman-data} converging to the
$\delta$-shock solution represents a justification of the concept
laid down in Section 2. In particular, observe that the vanishing viscosity
approximation is a special case of the weak asymptotic approximation
since the term $\eps u_{xx}$ is clearly of order ${\cal O}_{{\cal
D}'}(\eps)$.

To find the weak asymptotic solutions we need to find families of
smooth functions $(u_\eps)$, $(v_\eps)$, such that
\begin{equation}
\label{brioapp}
\begin{split}
\pa_t u_\eps + \pa_x\left(\Sfrac{u^2_\eps + v_\eps^2}{2}\right)  &=  {o}_{{\cal D}'}(1),       \\
\pa_t v_\eps + \pa_x\left(v_\eps(u_\eps - 1)\right) &= {o}_{{\cal
D}'}(1),
\end{split}
\end{equation}
\begin{equation}
\label{lim-brio} u_\eps\rightharpoonup u, \ \ v_\eps\rightharpoonup
v \ \ {\rm as} \ \ \eps\to 0,
\end{equation}
and such that $u(x,0) = U_0(x)$ and  $v(x,0) = V_0(x)$ are given by
\eqref{rieman-data}.
We shall prove the following theorem.
\begin{thm}
{\bf a)}
If $u_1\neq u_2$ then there exist weak
asymptotic solutions $(u_\eps)$, $(v_\eps)$ of the Brio system
\eqref{brio1}, such that the families $(u_\eps)$ and $(v_\eps)$ have
distributional limits
\begin{eqnarray}
 \label{BrioRiemann1-a}
u(x,t) & = & U_0(x-ct), \\
\label{BrioRiemann2-a}
v(x,t) & = & V_0(x-ct) + \alpha(t) \delta(x-ct),
\end{eqnarray}
where
\begin{equation}
\label{ec-a}c=\frac{u^2_1+v_1^2-u^2_2-v^2_2}{2(u_1-u_2)} \ \ {\rm
and} \ \ \alpha(t)= \frac{1}{2} \left(c(v_2-v_1)+(v_1(u_1-1)-v_2(u_2-1))\right)t.
\end{equation}
\end{thm}
\setcounter{thm}{0}
\begin{thm}
\label{th-3}
{\bf b)} If $v_1\neq v_2$ then there exist weak asymptotic solutions
$(u_\eps)$, $(v_\eps)$ of the Brio system \eqref{brio1}, such that
the families $(u_\eps)$ and $(v_\eps)$ have distributional limits
\begin{eqnarray}
\label{BrioRiemann1-b} u(x,t) & = & U_0(x-ct)+ \alpha(t)
\delta(x-ct), \\
\label{BrioRiemann2-b} v(x,t) & = & V_0(x-ct) ,
\end{eqnarray}
where
\begin{equation}
\label{ec-b}c=\frac{v_1(u_1-1)-v_2(u_2-1)}{v_1-v_2} \ \ {\rm and} \
\
\alpha(t)=\left(c(u_2-u_1)+\frac{u^2_1+v_1^2-u^2_2-v^2_2}{2}\right)t.
\end{equation}
\end{thm}

\begin{proof}
{\bf a)} Let $\rho\in C^\infty_c(\R)$ be an even, non-negative, smooth,
compactly supported function such that
$$
{\rm supp} \rho\subset (-1,1), \ \ \int_\R \rho(z)dz=1, \ \ \rho\geq
0.
$$ We take:
\begin{equation}
\label{aux-1}
\begin{split}
&  R_{\eps}(x,t)
       =\Sfrac{i}{\eps}\rho((x-ct-2\eps)/\eps)-\Sfrac{i}{\eps}\rho((x-ct+2\eps)/\eps), \\
&  \delta_{\eps}(x,t)=
\Sfrac{1}{\eps}\rho((x-ct-4\eps)/\eps)+\Sfrac{1}{\eps}\rho((x-ct+4\eps)/\eps).
\end{split}
\end{equation}
Next, define smooth functions $U_\eps$ and $V_\eps$ such that
\begin{align*}
U_\eps(x,t)&=\begin{cases} u_1, & x<ct-20\eps,\\
c+1, & ct-10\eps < x < ct+10\eps,\\
u_2, & x > ct+20\eps,
\end{cases}\\
V_\eps(x,t)&=\begin{cases} v_1, &  x<ct-20\eps,\\
0, & ct-10\eps < x < ct+10\eps,\\
v_2, & x > ct+20\eps.
\end{cases}
\end{align*}
Notice that
\begin{equation}
\label{aux-2}
  R_{\eps}\rightharpoonup 0,       \quad
  U_\eps R_{\eps} \rightharpoonup 0, \quad \mbox{ and } \quad
  U_\eps \delta_\eps \rightharpoonup 2 (c+1)\delta(x-ct).
\end{equation}
Moreover, we have
\begin{equation}
\label{aux-3}
V_\eps R_{\eps}\equiv 0, \quad
V_\eps \delta_\eps \equiv 0, \quad \mbox{ and } \quad
\delta_{\eps} R_{\eps} \equiv  0.
\end{equation}
Now make the ansatz
\begin{equation}
\label{new_1}
\begin{split}
u_\eps(x,t)&=U_\eps(x,t),\\
v_\eps(x,t)&=V_\eps(x,t) + \alpha(t)(\delta_{\eps}(x,t)+R_{\eps}(x,t)),
\end{split}
\end{equation}
and substitute it into equations \eqref{brioapp}.
Notice first of all that
\begin{align*}
v_\eps^2(x,t) = V_\eps^2 + \alpha^2(t)(R_{\eps}^2+\delta^2_\eps)
\end{align*}
by invoking \eqref{aux-3}.
Focusing on the expression $R_{\eps}^2+\delta^2_\eps$,
we take $\varphi\in C^\infty_c(\R)$ and consider the integral
\begin{align*}
&\int_{\R}(R_{\eps}^2+\delta^2_\eps)\varphi~dx \\
&=\int_{\R}\Sfrac{1}{\eps^2}\Big(-\rho^2((x-ct+2\eps)/\eps)-\rho^2((x-ct-2\eps)/\eps)\\
& \qquad\qquad
 + \rho^2((x-ct+4\eps)/\eps)+\rho((x-ct-4\eps)/\eps)\Big) \varphi~dx ={\cal O}(\eps).
\end{align*}
In the above reasoning, use was made of the following computation.
\begin{align*}
&\int_{\R}\Sfrac{1}{\eps^2}\left(\rho^2((x-ct+\alpha\eps)/\eps)
+\rho^2((x-ct-\beta\eps)/\eps)\right)\varphi(x)~dz
\\
&= \int_{\R}\Sfrac{1}{\eps}\rho^2(z)\left(
\varphi(ct+\eps(z-\alpha))+\varphi(ct+\eps(z+\beta)) \right)~dz\\
&=\int_{\R}\Sfrac{1}{\eps}\rho^2(z)\left(
2 \varphi(ct)+\eps\varphi'(ct)(\beta-\alpha)\right)dz +{\cal O}(\eps),
\quad \mbox{ for } \alpha, \beta \in \R.
\end{align*}
The last relation was found by making the changes of variables
$(x-ct+\alpha\eps)/\eps=z$ and $(x-ct-\beta\eps)/\eps=z $,
and observing that
$\int z\rho^2(z)dz=0$ since $\rho$ is an even function.
In the case at hand, we use $\alpha=\beta = 2$ for the first integral,
and  $\alpha=\beta = 4$ in the second integral.
Finally, it becomes plain that
\begin{equation}
\label{asym} v_\eps^2=V_\eps^2+ o_{{\cal D}'}(1).
\end{equation}
Therefore, taking into account Definition \ref{defwas}, from the
first equation in \eqref{brioapp}, we conclude that we need to check
whether
\begin{align*}
\pa_t U_\eps + \pa_x\frac{U_\eps^2+V_\eps^2}{2}=o_{{\cal D}'}(1),
\end{align*} and this reduces to
\begin{align}
\label{eq1} \int_0^T\left( -c[U] +
\Sfrac{1}{2}[U^2+V^2]\right)\varphi(ct,t)dt=o(1),
\end{align} where $[U]=u_2-u_1$ and
$[U^2+V^2]=u_2^2+v_2^2-u_1^2-u_2^2$.

However, this is indeed satisfied thanks to the choice of the
constant $c$ which was found from the Rankine-Hugoniot condition for
the first equation in \eqref{brio1}.

Let us now consider the second equation in \eqref{brioapp}.
First, notice that
\begin{equation*}
\begin{split}
&\pa_x (v_\eps(u_\eps-1)) = \pa_x\left(U_\eps V_\eps
     + (c+1)\alpha(t)\delta_\eps-V_\eps-\alpha(t)\delta_\eps \right) + o_{{\cal D}'}(1)\\
&\qquad\qquad\qquad= (v_1(1-u_1)+v_2(u_2-1)\delta(x-ct)+ c \alpha(t)\delta'(x-ct))
 + o_{{\cal D}'}(1).
\end{split}
\end{equation*}
Next, note also that
\begin{equation*}
\pa_tv_\eps=-c(v_2-v_1)\delta(x-ct)+\alpha'(t) \delta(x-ct) -c \alpha(t)
\delta'(x-ct)+o_{{\cal D}'}(1).
\end{equation*}
Adding the latter two expressions, we obtain
\begin{align*}
&\pa_tv_\eps+\pa_x\left(v_\eps(u_\eps-1)\right)\\&
=\left(-c(v_2-v_1)+
        \alpha'(t) +  (v_1(1-u_1)+v_2(u_2-1) \right)\delta(x-ct)+o_{{\cal D}'}(1).
\end{align*}
From here, we conclude that choosing $\alpha$ as given in
\eqref{ec-a}, the first equation in \eqref{brioapp} is satisfied, as
well. This concludes the proof of part (a).

\vspace{0.5cm}
\noindent
{\bf b)} In this case, an appropriate weak asymptotic solution is given by
\begin{equation*}
\begin{split}
u_\eps(x,t)&=U_\eps(x,t) + \alpha(t)(\delta_{\eps}(x,t)
                      +  R_{1\eps}(x,t)) + \sqrt{2 c \alpha(t)} R_{2\eps}(x,t)),\\
v_\eps(x,t)&=V_\eps(x,t),
\end{split}
\end{equation*}
where
\begin{equation*}
c=\frac{v_1u_1-v_2u_2}{v_1-v_2}-1 \ \ {\rm and} \ \
\alpha(t)=\left(c(u_1-u_2)-\frac{u_1^2+v_1^2-u_2^2-v_2^2}{2}\right)t,
\end{equation*}and
\begin{align*}
U_\eps(x,t)&=\begin{cases} u_1, & x<ct-20\eps\\
0, & ct-10\eps < x < ct+10\eps\\
u_2, & x > ct+20\eps
\end{cases}\\
V_\eps(x,t)&=\begin{cases} v_1, & x<ct-20\eps\\
0, & ct-10\eps < x < ct+10\eps\\
v_2, & x > ct+20\eps
\end{cases};
\end{align*}
\begin{equation*}
\begin{split}
&R_{1\eps}(x,t)=\frac{i}{\eps}\rho((x-ct-2\eps)/\eps)-\frac{i}{\eps}\rho((x-ct+2\eps)/\eps);
\\
&R_{2\eps}(x,t)=\frac{1}{\sqrt{\eps}} \left[ \rho((x-ct)/\eps) \right]^{\frac{1}{2}};
\\
&\delta_{\eps}(x,t)=
\frac{1}{\eps}\rho((x-ct-4\eps)/\eps)+\frac{1}{\eps}\rho((x-ct+4\eps)/\eps),
\end{split}
\end{equation*}
where
$\rho$ is the same smooth non-negative even function as used in the
previous examples. The proof then follows the ideas of the proof of (a).
\end{proof}

An important corollary (to be used in the Appendix) of the proof of
the previous theorem is that it gives another interesting class of weak
asymptotic solutions to \eqref{brio1} having the $\delta$
distribution as their limit.
\begin{corollary}
\label{only}
If $u_1=u_2$ and $v_1^2=v_2^2$ then for any $c\in \R$
the families $(u_\eps)$ and $(v_\eps)$ given by \eqref{new_1} are
the weak asymptotic solution to \eqref{brio1}, \eqref{rieman-data}.
\end{corollary}
\begin{proof}
It is enough to notice that \eqref{eq1} is satisfied independently
on $c$ since $[U]=[U^2+V^2]=0$.
\end{proof}

To close the section, we should mention that while the extension of the weak
asymptotics method to complex-valued solutions was crucial for finding a solution
of the system \eqref{brio1}, it might not be appropriate in other contexts
as it might lead to strong non-uniqueness.
For example, using complex-valued weak asymptotic solutions of similar form
as \eqref{BrioRiemann1-b} for the inviscid
Burgers equation, one may construct a family of distinct solutions emanating
from the same initial data, all of which also satisfy
the Lax admissibility condition.

\section{Generalized weak solutions for the Brio system and the uniqueness issue}

By comparing Theorem \ref{thm-cnl} and Theorem \ref{th-3}, we see
that the limit distributions $u$ and $v$ given in Theorem \ref{th-3}
represent $\delta$-shock solutions to \eqref{brio1} with initial
data $u(x,0) = U_0(x)$ and $v(x,0) = V_0(x)$. However, we want to
incorporate such solutions into the Lax admissibility concept and
this is the goal in this section.

We focus on Definition \ref{def-gen} b) from \cite{DSH} where the fluxes $f$ and $g$
are given by the Brio system. The same can be
done with Definition \ref{def-gen} a) but it appears that the
solutions which it generates do not fit into the Lax
admissibility concept. More details of this case will be provided in the
Appendix.

Recall that in the case $v_1>0>v_2$ there exists no Lax-admissible
solution to the Riemann problem \eqref{brio1}, with the Riemann
initial data $U_0$ and $V_0$ given by \eqref{rieman-data} (see
\cite{Lfl2}). If $v_1$ and $v_2$ do not satisfy this relation, we
have the classical Lax admissible solution to the appropriate
Riemann problem consisting of the elementary waves, i.e. shock and
rarefaction waves.
For $L^{\infty}$-small data, such solution is unique since the
system is genuinely nonlinear for $v>0$ and $v<0$. Theorem
\ref{thm-cnl} states that we can also have $\delta$-shock wave
solutions, but as Proposition \ref{non-unique} shows, there is
strong non-uniqueness. In order to eliminate at least some of
solutions which are inconsistent with the physical intuition, we
shall use the Lax compressivity conditions for the $\delta$-shock
wave. In order to introduce them, let us recall that the
characteristic velocities for the Brio system are \cite{Lfl2}:
$$
\lambda_1(u,v)=u-1/2-\sqrt{1/4+v^2}, \ \
\lambda_2(u,v)=u-1/2+\sqrt{1/4+v^2}.
$$
The corresponding rarefaction waves are given by
\begin{align*}
RW_1: \; u=-\frac{1}{2}\left(
\sqrt{4v^2+1}-log(1+\sqrt{4v^2+1})\right)+C_1;\\
RW_2: \; u=\frac{1}{2}\left(
\sqrt{4v^2+1}+log(-1+\sqrt{4v^2+1})\right)+C_2;
\end{align*} The shock waves are given by
$$
SW_{1,2}: \; u-u_1=\frac{v-v_1}{v+v_1}\left(1\mp \sqrt{(v+v_1)^2+1}
\right).
$$
A phase space picture for a given left and right state is shown in Figure 1.
\begin{figure}[htp]
\begin{center}
  \includegraphics[width=4in]{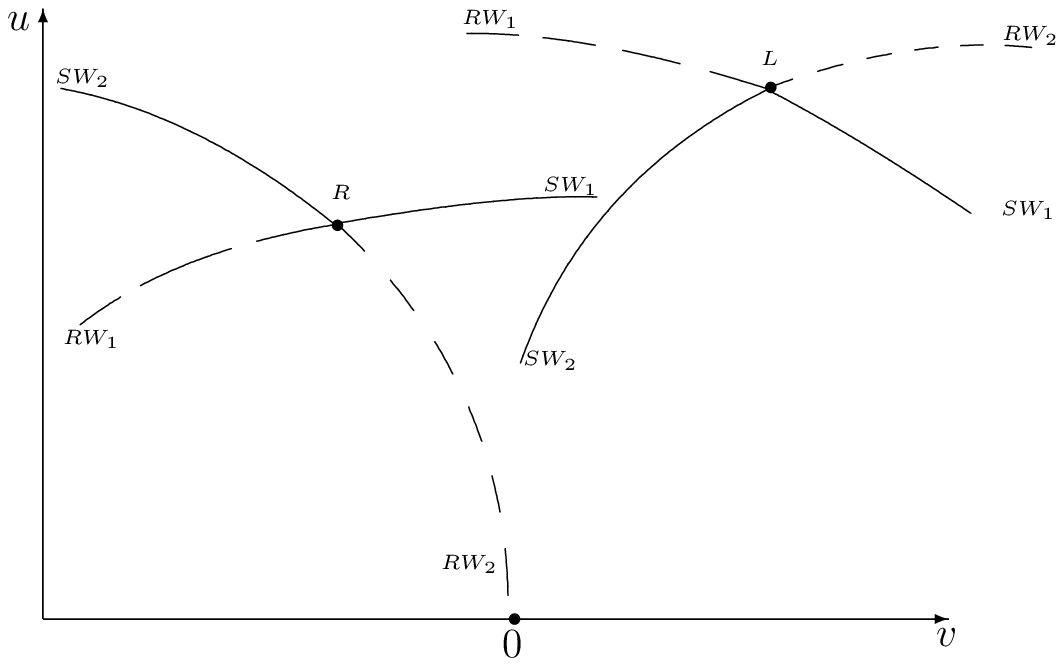}\\
  \caption{$L$ is the left, and $R$ is the right state.}
  \label{fig1}
  \end{center}
\end{figure}
%
%

The following definition introduces a compressivity demand on the
characteristics of \eqref{brio1} meaning that the characteristics
enters the $\delta$-shock from both sides. It is standard for the
classical shock waves and they are known as Lax admissibility
conditions. Note the usual demand on the $\delta$-shock wave is an
overcompressivity condition demanding that both characteristic field
$\lambda_1$ and $\lambda_2$, satisfy \eqref{v-f} below
\cite{DM2,DSH, ercole, KK1, TZZ}. However, we were not able to find
solutions involving overcompressive $\delta$-shocks and we confine
ourselves on a less restrictive demand which still includes
concentration effects. The definition concerning the admissible
$\delta$-shock solutions of \eqref{brio1} such as defined in Theorem
3.1 follows.
\begin{definition}
\label{admis}
A $\delta$-shock solution of \eqref{brio1},
connecting a left state $L = (u_1,v_1)$ and a right state $R = (u_2,v_2)$
is $i$-admissible if
\begin{equation}
\label{v-f} \lambda_i(u_2,v_2) \leq c \leq \lambda_i(u_1,v_1),
\end{equation}
for $i=1$ or $i=2$. For such $\delta$ shock wave we say that it is
compressive.
\end{definition}

Thus for a general Riemann problem, one may say that a solution of
\eqref{brio1}, \eqref{rieman-data} which contains a $\delta$-shock
wave is admissible if it consists of a combination of the classical
Lax admissible simple waves (shock or rarefaction) and compressive
$\delta$ waves.

The following lemma will be crucial for existence of admissible
$\delta$-shock solutions to Riemann problems corresponding to
\eqref{brio1}.

\begin{lemma}
\label{l-v} Assume that the initial data \eqref{rieman-data} are
such that $u_1=u_2=\tilde{u}$, $v_1=0$ and $v_2<0$. Then, the
$\delta$-shock solution
\begin{equation}
\label{even-v}
\begin{split}
u(x,t)= & \tilde{u} + \alpha(t) \delta(x-ct), \\
v(x,t)= & 0,
\end{split}
\end{equation}
where $\alpha(t)$ and $c$ are given by \eqref{ec-b},
represents $1$-admissible $\delta$-shock solution.
\end{lemma}
\begin{proof}
The functions given by \eqref{even-v} represent $\delta$ shock
solution to \eqref{brio1}, \eqref{rieman-data} according to Theorem
\ref{thm-cnl}, b). In order to prove that the solution is
$1$-admissible, recall that
$c=\frac{v_2(u_2-1)-v_1(u_1-1)}{v_2-v_1}$. Then, due to \eqref{v-f},
we need to show:
\begin{equation*}
\begin{split}
\lambda_1(u_2,v_2)&=u_2-1/2-\sqrt{1/4+v_2^2} \leq
\frac{v_2(u_2-1)-v_1(u_1-1)}{v_2-v_1} \\&\leq
u_1-1/2-\sqrt{1/4+v_1^2}=\lambda_1(u_1,v_1).
\end{split}
\end{equation*} Since $u_1=u_2=\tilde{u}$ and $v_1=0$, the latter reduces to
\begin{equation*}
\tilde{u}-1/2-\sqrt{1/4+v_2^2} \leq \tilde{u}-1 \leq \tilde{u}-1
\Rightarrow 1/2-\sqrt{1/4+v_2^2} \leq 0,
\end{equation*} which is clearly true. This concludes the proof.
\end{proof}

With this lemma established, we can attempt the proof of
the following theorem.

\begin{thm}
Given any Riemann initial data \eqref{rieman-data} such that
$v_2<0<v_1$, there exists a solution of \eqref{brio1} in the sense
of Definition \ref{def-gen} which consists of a combination of the
classical Lax admissible simple waves (shock or rarefaction) and
compressive $\delta$ waves, $1$-admissible in the sense of
Definition \ref{admis}.
\end{thm}
\begin{proof}
The solution is plotted on Figure \ref{sol-v}. First, we have the
rarefaction wave--$1$ (RW1) issuing from the left state $L=(u_1,v_1)$
and connecting it to the state $(u_m,0)$. Then, we connect the state
$(u_m,0)$ with the state $(u_m, v_m)$  by the $\delta$-shock wave,
and finally we connect $(u_m, v_m)$ with $R=(u_2,v_2)$ by the shock
wave--$2$ (SW2) or rarefaction wave--$2$ (RW2).

The solution is admissible since all the simple shocks which it contains
are admissible. Namely, Lemma \ref{l-v} provides
admissibility for the $\delta$-shock wave while other waves are
admissible according to the standard theory (see Figure \ref{fig1}).
Furthermore, such combination of shocks is clearly possible since
the speed of the state $L$ equals to $\lambda_1(u_1,v_1)$ and it is
less than the speed $\lambda_1(u_m,0)$ of the middle point $(u_m,0)$
(since they are connected by the rarefaction wave). Furthermore, the
speed of the $\delta$ shock connecting $(u_m,0)$ and $(u_m,v_m)$
equals $v_m(u_m-1)/v_m=\lambda_1(u_m,0)$ and it is slower than
the speed of the state $(u_m,v_m)$ which equals either
$\lambda_2(u_m,v_m)$ (if we have RW2 between $(u_m,v_m)$ and
$(u_2,v_2)$) or we have
$c=\frac{v_2(u_2-1)-v_m(u_m-1)}{v_2-v_1}<\lambda_1(u_m,0)$ (if we
have SW2 between $(u_m,v_m)$ and $(u_2,v_2)$). Both situations are
plotted in the phase space picture shown in Figure \ref{phase}.

Indeed, if $(u_m,v_m)$ is connected to $R=(u_2,v_2)$ by the RW2,
then the speed of $(u_m,v_m)$ is
$\lambda_2(u_m,v_m)>\lambda_1(u_m,0)$. On the other hand, if
$(u_m,v_m)$ is connected to $R=(u_2,v_2)$ by the SW2, then its speed
is
$$
\frac{v_2(u_2-1)-v_m(u_m-1)}{v_2-v_m}=u_m-1+v_2\frac{u_2-u_m}{v_2-v_m}>\lambda_1(u_m,0)=u_m-1,
$$since $v_2<0$, $v_2-v_m>0$, and $u_2<u_m$ (see Figure
\ref{sol-v}).
\end{proof}

\begin{figure}[htp]
\begin{center}
  \includegraphics[width=4in]{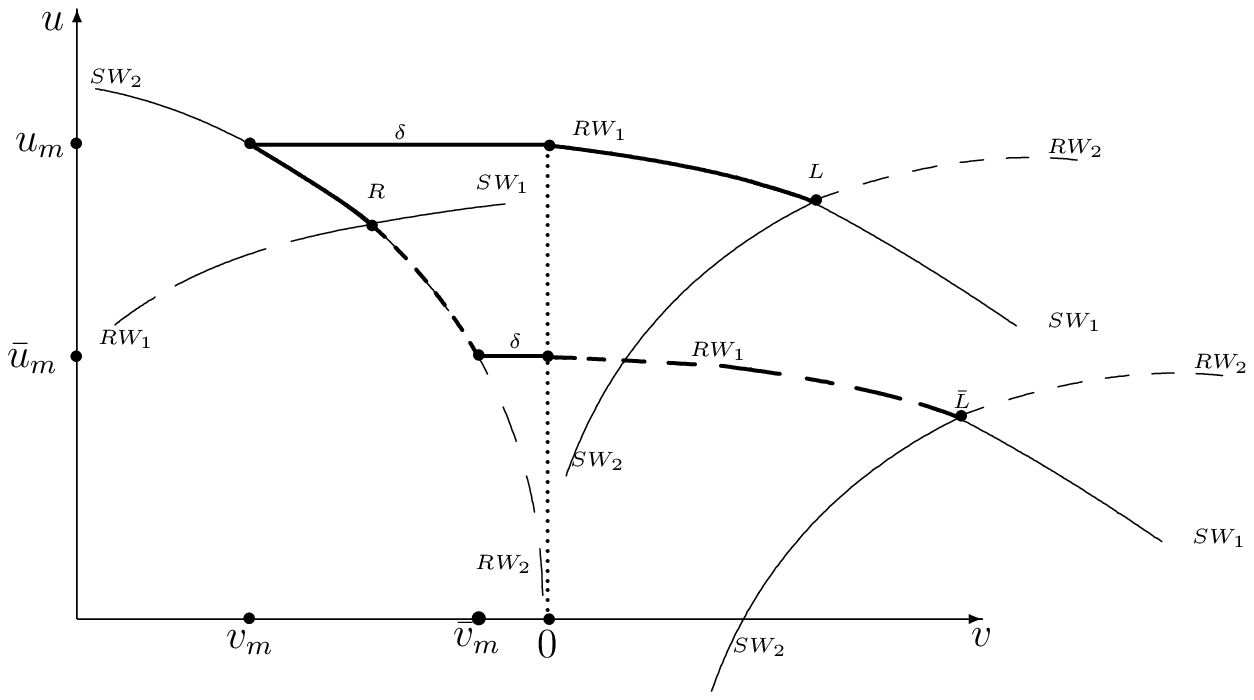}\\
  \caption{Thick full line: $L \overset{RW1}{\to}
(u_m,0) \overset{\delta}{\to} (u_m,v_m) \overset{SW2}{\to} R$. Thick
dashed line: $\bar{L} \overset{RW1}{\to} (\bar{u}_m,0)
\overset{\delta}{\to} (\bar{u}_m,\bar{v}_m) \overset{RW2}{\to} R$.}
  \label{sol-v}
  \end{center}
\end{figure}


\begin{figure}[htp]
\begin{center}
  \includegraphics[width=4in]{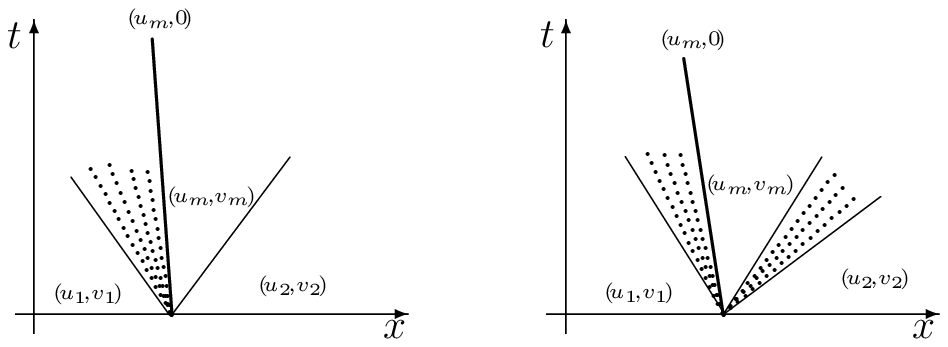}\\
  \caption{The left plot corresponds to the situation
when $(u_m,v_m)$ is connected with $(u_2,v_2)$ by SW2, and the right
plot when $(u_m,v_m)$ is connected with $(u_2,v_2)$ by RW2.}
  \label{phase}
  \end{center}
\end{figure}


This theorem provides existence of an admissible $\delta$-shock
solution of the system \eqref{brio1} with Riemann data \eqref{rieman-data}.
However, even with the admissibility concept provided by Definition \ref{admis}.
it is not difficult to see that uniqueness may not hold.
For example,
a left state $L=(u_1,v_1)$ and a right state $R=(u_2,v_2)$ may be joined
directly by a $1$-admissible $\delta$ shock as long as
%
%
\begin{equation*}
v_1\frac{u_2-u_1}{v_2-v_1} \geq
\frac{1}{2}-\sqrt{\frac{1}{4}+v_2^2}, \ \ {\rm and}  \ \
v_2\frac{u_2-u_1}{v_2-v_1} \leq
\frac{1}{2}-\sqrt{\frac{1}{4}+v_1^2},
\end{equation*}
and this is true whenever $u_1-u_2$ is large enough and
$v_2<0<v_1$. We could, of course, add certain conditions which would
eliminate the non-uniqueness. For instance, we could announce a
$\delta$ shock as admissible only if it connects states $L=(u,0)$
and $R=(u,v)$, $v<0$. However, we do not have any physical
justification for such condition and we shall confine ourselves on
the existence statement.

\section{Appendix}

We shall end the paper by considering the possibility of
the $\delta$ distribution to be adjoint to the unknown $v$ in
\eqref{brio1}.
We start with a lemma which will help us to connect certain
states by admissible $\delta$ shocks residing in the unknown $v$.
\begin{lemma}
\label{l-new-1}
Assume that in \eqref{rieman-data} we have $u_1=u_2$ and $v_1=-v_2>0$.
Then, if
$c=\lambda_i(u_1,v_1)=\lambda_i(u_2,v_2)$, $i=1,2$, the functions
\begin{equation}
\label{even}
\begin{split}
u(x,t)= & U_0(x-ct), \\
v(x,t)= & V_0(x-ct) + \alpha(t) \delta(x-ct),
\end{split}
\end{equation}
where $\alpha(t)$ is given by \eqref{ec-a},
represent the $i$-admissible $\delta$-shock solution to
\eqref{brio1}.
\end{lemma}
\begin{proof}
It is enough to rely on Corollary \ref{only} and proof of Theorem
\ref{thm-cnl}. Indeed, taking \eqref{even} for any $c\in \R$, and
inserting them into Definition \ref{admis}, we see that such $u$ and $v$
represent the $\delta$-shock solution to \eqref{brio1},
\eqref{rieman-data}. To see this, one may use the same reasoning
as in the proof of Theorem \ref{thm-cnl} and relation \eqref{eq1}.

Next, we take $c=\lambda_1(u_1,v_1)=\lambda_1(u_2,v_2)$
or
$c=\lambda_2(u_1,v_1)=\lambda_2(u_2,v_2)$
to conclude that the pair $(u,v)$ is $1$-admissible or $2$-admissible, respectively,
in the sense of Definition \ref{admis}.
\end{proof}

Using the lemma, we can connect the states
$$L=(u_1,v_1) \ \  {\rm and} \ \  R=(u_2,v_2) \ \  \text{where} \ \ u_2>u_1$$
by an admissible $\delta$-shock solution $(u,v)$ to \eqref{brio1},
\eqref{rieman-data} admitting the $\delta$ shock in the function $v$
through one of the following procedures

\begin{itemize}

\item $L \to (v_m,u_m)$ by RW1; $(v_m,u_m)\to (-v_m,u_m)$ by the
$\delta$ shock with the speed $c=\lambda_1(u_m,v_m)$; $(-v_m,u_m)\to
R$ by RW2. See Figure \ref{old-2}.  In this case, we can also put
$c=\lambda_2(u_m,v_m)$. If we take such $c$ then $\delta$ shock
travels with the state $(u_m,-v_m)$. If we take
$c=\lambda_1(u_m,v_m)$ then $\delta$ shock travels with the state
$(u_m,v_m)$. Remark the non-uniqueness that we have here.

\item $L \to (v_m,u_m)$ by SW1; $(v_m,u_m)\to (-v_m,u_m)$ by the
$\delta$ shock with the speed $c=\lambda_2(u_m,v_m)$; $(-v_m,u_m)\to
R$ by RW2. See Figure \ref{old-3}.

\item $L \to (v_m,u_m)$ by RW1; $(v_m,u_m)\to (-v_m,u_m)$ by the
$\delta$ shock with the speed $c=\lambda_1(u_m,v_m)$; $(-v_m,u_m)\to
R$ by SW2. See Figure \ref{old-4}.

\end{itemize}

\begin{figure}[htp]
\begin{center}
  \includegraphics[width=4in]{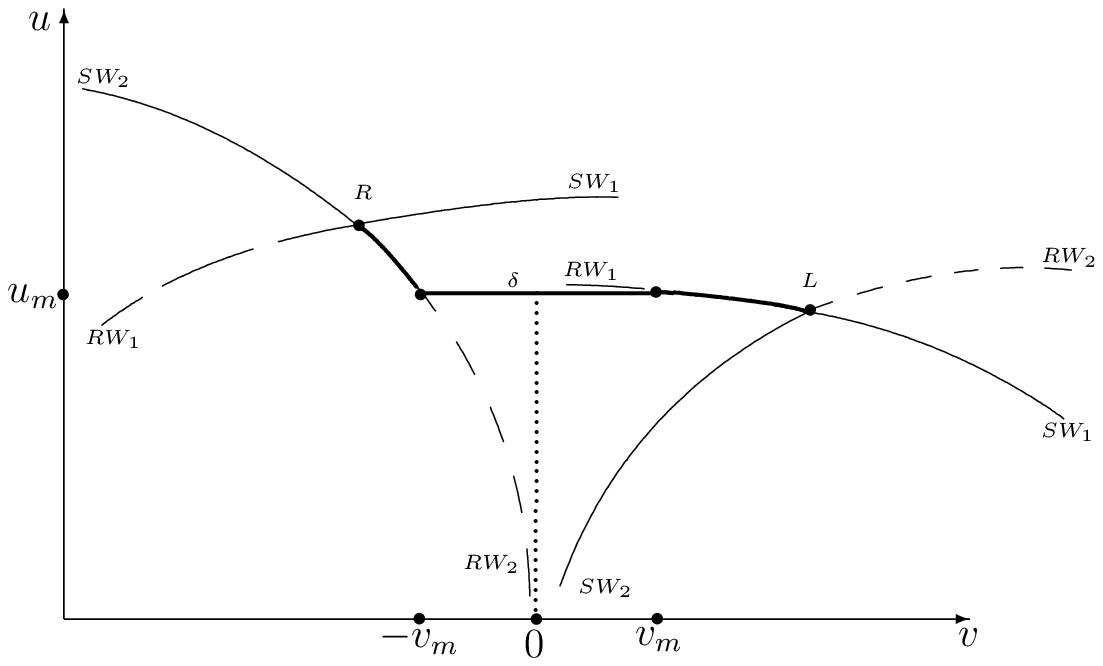}\\
  \caption{$L=(u_1,v_1) \overset{RW1}{\to}
(u_m,v_m)\overset{\delta}{\to} (u_m,-v_m)\overset{RW2}{\to}
R=(u_2,v_2)$.}
  \label{old-2}
  \end{center}
\end{figure}


\begin{figure}[htp]
\begin{center}
  \includegraphics[width=4in]{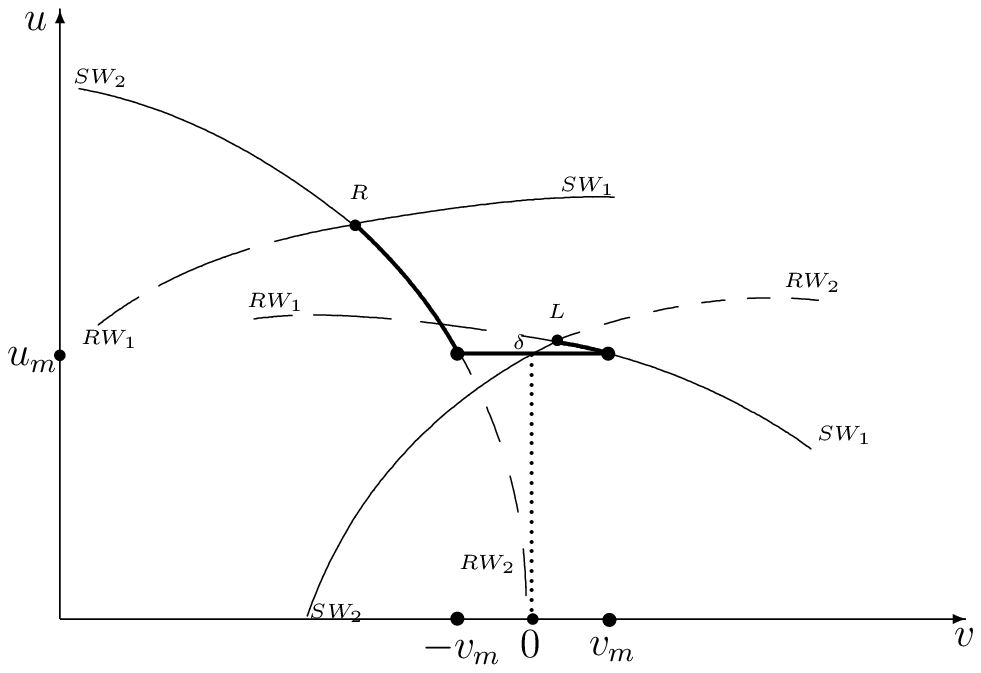}\\
  \caption{$L=(u_1,v_1) \overset{SW1}{\to}
(u_m,v_m)\overset{\delta}{\to} (u_m,-v_m)\overset{RW2}{\to}
R=(u_2,v_2)$.}
  \label{old-3}
  \end{center}
\end{figure}


\begin{figure}[htp]
\begin{center}
  \includegraphics[width=4in]{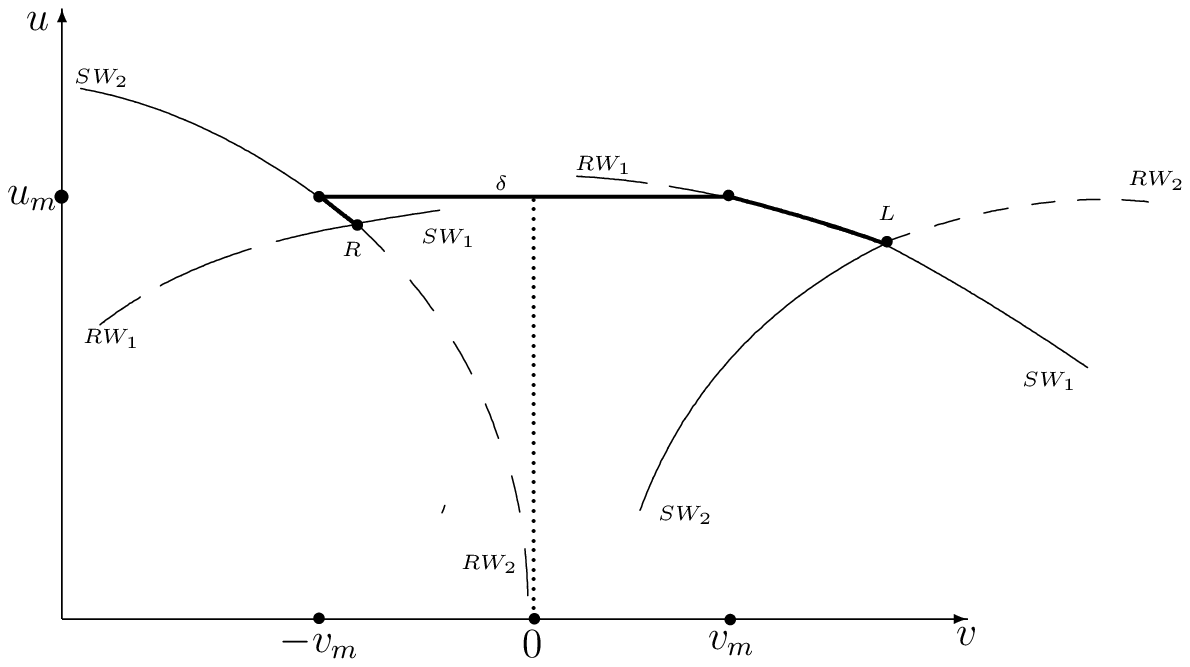}\\
  \caption{$L=(u_1,v_1) \overset{RW1}{\to}
(u_m,v_m)\overset{\delta}{\to} (u_m,-v_m)\overset{SW2}{\to}
R=(u_2,v_2)$.}
  \label{old-4}
  \end{center}
\end{figure}


In the case when $u_2<u_1$, we do not have a general recipe for
connecting the states $L=(u_1,v_1)$ and $R=(u_2,v_2)$ by an
admissible $\delta$ shock solution with the $\delta$ function
adjoined  to $v$. Finally, observe that each of the $\delta$ shocks
in this section is not really compressive since characteristics from
both sides of the shock are parallel to the shock. Thus we cannot
say that concentration effects are present.

\newpage

{\bf Acknowledgements.} We would like to thank to Professor Marko
Nedeljkov for pointing our attention to the Brio system, as well as
for his valuable comments and remarks. We are grateful to the anonymous
referee for carefully reading the manuscript, and for providing helpful
suggestions for improvement of the article.

Darko Mitrovic is employed as a part-time postdoctoral fellow at the University
of Bergen, funded by the project "Modeling transport in porous
media over multiple scales" of the Research Council of Norway.

\end{document}